\documentclass[12pt]{amsart}
\pdfoutput=1
\usepackage{amsmath,amssymb,amscd,graphicx,verbatim,float}
\usepackage[colorlinks=true, pdfstartview=FitV, linkcolor=blue, citecolor=blue, urlcolor=blue]{hyperref}
\usepackage{mathpazo}
\theoremstyle{plain}
\newtheorem{theorem}{Theorem}

\newtheorem*{proposition}{Proposition}

\newcommand{\tre}{\text{Re}}
\newcommand{\tim}{\text{Im}}

\theoremstyle{definition}
\newtheorem*{remark}{Remark}

\begin{document}
\title{Notes on $\log(\zeta(s))^{\prime\prime}$}
\author{Jeffrey Stopple}
\begin{abstract}
Motivated by the connection to the pair correlation of the Riemann zeros, we investigate the second derivative of the logarithm of the Riemann zeta function, in particular the zeros of this function. 
Theorem 1 gives a zero-free region.  Theorem 2 gives an asymptotic estimate for the number of nontrivial zeros to height $T$.  Theorem 3 is a zero density estimate.
\end{abstract}
\email{stopple@math.ucsb.edu}\address{Mathematics Department, UC Santa Barbara, Santa Barbara CA 93106}

\maketitle

Bogomolny and Keating  \cite{BoKe} were the first to observe that the function
%$(\zeta(s)\zeta^{\prime\prime}(s)-\zeta^\prime(s)^2)/\zeta(s)^2$ 
$(\zeta^\prime(s)/\zeta(s))^\prime$
appears in the pair correlation for the Riemann zeros\footnote{See also the recent work of Rodgers \cite{BRodgers}, and Ford and Zaharescu \cite{FZ}.}.   Berry and Keating  \cite{BK} wrote in that context
\begin{quote}
\emph{\lq\lq The appearance of $\zeta(s)$ indicates an astonishing resurgence property of the zeros: in the pair correlation of high Riemann zeros, the low Riemann zeros appear as resonances.\rq\rq}
\end{quote}
There has been extensive investigation of the zeros of $\zeta^\prime(s)$ and their connection to the Riemann Hypothesis, via the logarithmic derivative $\zeta^\prime/\zeta(s)$.  However there seems to be nothing in the literature about the zeros of the derivative
\[
\log(\zeta(s))^{\prime\prime}=\left(\frac{\zeta^\prime(s)}{\zeta(s)}\right)^\prime=\frac{\zeta(s)\zeta^{\prime\prime}(s)-\zeta^\prime(s)^2}{\zeta(s)^2}.
\]
The connection to the pair correlation of the Riemann zeros is  motivation for further study.  

Further motivation comes from Montgomery's review in \emph{Math.\ Reviews} of Levinson \cite{Levinson}, in which he says \lq\lq The author's method can be applied to functions other than $G(s)$, and in particular one may use differential operators of higher order. Whether sharper results can be obtained in this manner remains to be seen.\rq\rq

\begin{figure*}
\begin{center}
\includegraphics[scale=.8, viewport=0 0 400 250,clip]{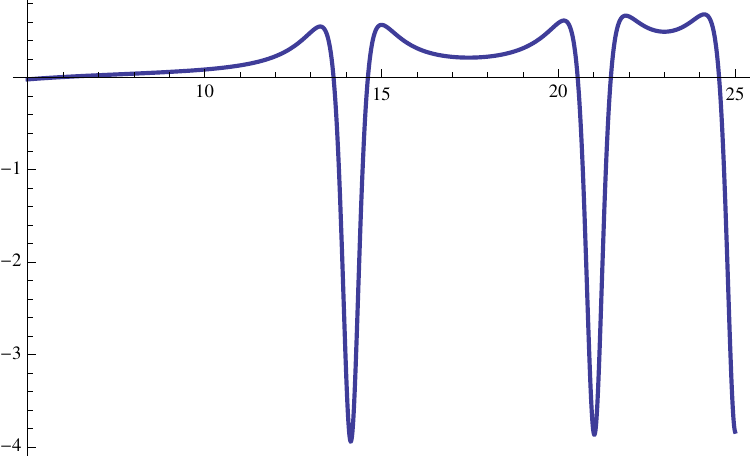}
\caption{$\tre((\zeta^\prime/\zeta)^\prime(1+it))$ is the resurgent contribution of $\zeta(s)$ to pair correlation.}\label{F:0}
\end{center}
\end{figure*}

\subsubsection*{Notation}  We let
\[
\nu(s)=\zeta(s)\zeta^{\prime\prime}(s)-\zeta^\prime(s)^2.
\]

\subsection*{Elementary facts}
Near $s=1$,
\[
\log(\zeta(s))^{\prime\prime}=\frac{1}{(s-1)^2}+O(1).
\]
Near a zero $\rho$ of $\zeta(s)$ of order $n_\rho$,
\[
\log(\zeta(s))^{\prime\prime}=\frac{-n_\rho}{(s-\rho)^2}+O(1),
\]
and so $\nu(s)$ has a zero of order $2n_\rho-2$.
In particular for a simple zero of $\zeta(s)$, this tells us that $\nu(\rho)\ne 0$.
There are no other poles.  The zeros of $\log(\zeta(s))^{\prime\prime}$ are the zeros of $\nu(s)$, exclusive of any possible multiple zeros of $\zeta(s)$.  

We have that, for $\tre(s)>1$, 
\begin{equation}\label{Eq:an1}
\nu(s)
=\sum_n\left(\sum_{d|n}\log(d)^2-\log(d)\log(n/d)\right)n^{-s}.
\end{equation}
With $\Lambda(n)$ Von Mangoldt's function, and $\tau(n)$ the divisor function we have that
\[
\log(\zeta(s))^{\prime\prime}=\sum_n \Lambda(n)\log(n) n^{-s},\qquad
\zeta(s)^2=\sum_n \tau(n)n^{-s}.
\]
Thus we also have that
\begin{equation}\label{Eq:an2}
\nu(s)
=\sum_n \left(\sum_{d|n}\Lambda(d)\log(d)\tau(n/d)\right)n^{-s}.
\end{equation}
We will let $a(n)$ denote the Dirichlet series coefficients of $\nu(s)$, given by either \eqref{Eq:an1} and \eqref{Eq:an2}.  Let
\[
A(x)=\sum_{n<x} a(n).
\]
We have that for $c>1$,
\[
A(x)=\frac{1}{2\pi i}\int_{c-i\infty}^{c+i\infty}
\nu(w)\frac{x^w}{w}\, dw.
\]
Moving the contour past the pole at $s=1$, we have that for $0<c<1$
\begin{equation}\label{Eq:integral}
A(x)=
x\cdot p(\log(x))+\frac{1}{2\pi i}\int_{c-i\infty}^{c+i\infty}
\nu(w)\frac{x^w}{w}\, dw,
\end{equation}
where
\[
p(t)=\frac{t^3}{6}+\left(C_0 -\frac{1}{2}\right)t^2+\left(1-4 C_1-2C_0 \right)t+4 C _2+4 C_1+2C_0 -1,
\]
with $C_0$ is the Euler constant, and $C_1$ and $C_2$ are Stieltjes constants.  With $p(t)$ as above one can show by Euler MacLaurin Summation \cite[Appendix B]{MV} and  the \lq method of the hyperbola\rq\ \cite[(2.9)]{MV} that
\begin{equation}\label{Eq:1}
A(x)=
x\cdot p(\log(x))+O\left(x^{1/2}\log(x)^2\right),
\end{equation}
i.e., the integral in \eqref{Eq:integral} is $O(x^{1/2}\log(x)^2)$.  

\subsection*{Functional Equation}
As usual let
\begin{align*}
\chi(s)=&2 (2\pi)^{s-1}\sin(\pi s/2)\Gamma(1-s)\\
=&\frac{\pi^{(s-1)/2}\Gamma((1-s)/2)}{\pi^{-s/2}\Gamma(s/2)}.
\end{align*}
Differentiating the functional equation 
$
\zeta(s)=\chi(s)\zeta(1-s)
$
we deduce that
\begin{multline}\label{Eq:FE}
\nu(s)=\\
\chi^2(s) \left(\nu(1-s)
+\left(\psi^\prime(1-s)-(\pi/2)^2\csc(\pi s/2)^2 \right) \zeta (1-s)^2\right).
\end{multline}
Here $\psi^\prime(s)$ denotes the derivative of the \textsc{digamma} function
\[
\psi(s)=\frac{\Gamma^\prime(s)}{\Gamma(s)}.
\]
Stirling's Formula tells us that as $s\to\infty$ in the region $|\arg(s)|\le \pi-\delta$,
\[
\psi^\prime(s)=1/s+O(1/s^2).
\]
As $t\to\infty$ we have that for $\sigma>a$ fixed,
\begin{gather}
\chi^2(s)\ll t^{1-2\sigma}\\
\chi^2(s)\left(\psi^\prime(1-s)-(\pi/2)^2\csc(\pi s/2)^2\right) \ll t^{-2\sigma}.
\end{gather}
Thus as $s\to\infty$ in the region $|\arg(s)|\le \pi-\delta$,
\begin{equation}
\nu(s)=
\begin{cases}\label{Eq:6}
O(1)&\sigma\ge 1+\delta>1\\
O(t^{1-2\sigma})&\sigma\le-\delta<0.
\end{cases}
\end{equation}

From the functional equation
\[
\zeta(1-s)=2(2\pi)^{-s}\cos(\pi s/2)\Gamma(s)\zeta(s),
\]
we deduce
\begin{equation}
\log(\zeta(1-s))^{\prime\prime}=-\frac{\pi^2}{4}\sec^2(\pi s/2)+\psi^\prime(s)+\log(\zeta(s))^{\prime\prime}.
\end{equation}

\subsection*{Asymptotics}
With $a(n)\ll n^\epsilon$ we can estimate the sum of the series for $n\ge 3$ to obtain
\[
\log(\zeta(s))^{\prime\prime}=\frac{\log(2)^2}{2^s}+O\left(\frac{\exp(-\sigma)}{1+\epsilon-\sigma}\right)\quad\text{for }\sigma>1+\epsilon.
\]
Now $|\sec^2(\pi s/2)|\ll \exp(-\pi t)$.  Thus we have
\begin{proposition}  As $s\to \infty$ in a vertical strip $1+\epsilon< \sigma<\sigma_0$,
\begin{equation}
\log(\zeta(1-s))^{\prime\prime}=\frac{\log(2)^2}{2^{s}}+O\left(\frac{\exp(-\sigma)}{1+\epsilon-\sigma}\right)+O\left(\frac{1}{s}\right).
\end{equation}
On the other hand, if $t\to\infty$ with $|s|^2<2^\sigma$, then
\begin{equation}\label{Eq:otoh}
\log(\zeta(1-s))^{\prime\prime}=\frac{1}{s}+O\left(\frac{1}{s^2}\right).
\end{equation}
\end{proposition}
%\begin{comment}
\begin{figure}[H]
\begin{center}
\includegraphics[scale=.1, viewport=0 0 3300 5000]{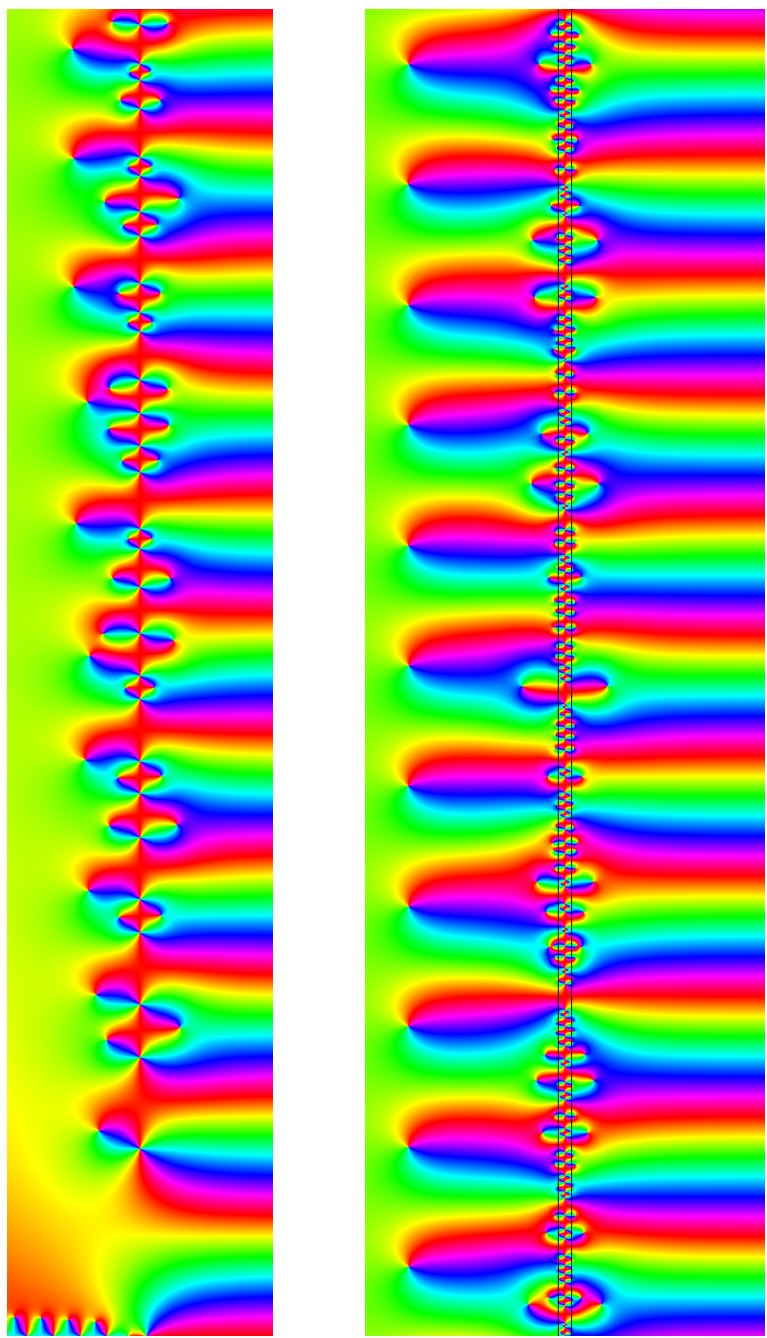}
\caption{Argument of $\log(\zeta(s))^{\prime\prime}$.  On the left, the vertical strip $-9.5\le \sigma\le 10.5$, and $0\le t\le 100$.  On the right, $-14.5\le \sigma\le 15.5$, and $10^4\le t\le 10^4+100$.  The dotted lines denote $\sigma=0$ and $\sigma=1$. }\label{F:2}
\end{center}
\end{figure}
%\end{comment}

On the border of these two asymptotic regimes, we will see a cancellation where
\[
\frac{1}{s}\approx \frac{-\log(2)^2}{2^s},
\]
creating zeros of $\nu(s)$  which we refer to as \textsc{asymptotically trivial of the first kind}.  Equating modulus and argument, this happens when
\begin{gather*}
2^\sigma\approx \log(2)^2\left(\sigma^2+t^2\right)^{1/2}\quad\text{or}\quad \sigma\approx\log(t)/\log(2),\\
\text{and also}\qquad\tan(t\log(2))\approx t/\sigma.
\end{gather*}

With $\sigma$ and $t$ positive, both $\cos(t\log(2))$ and $\sin(t\log(2))$ need to be negative, and since $\sigma$ is very small compared to $t$ we deduce that $t\log(2)$ is slightly larger than $2\pi n+3\pi/2$ for integer $n$; i.e.\ the imaginary part is about $9.1n+6.8$.  The real part is near $1-\log(t)/\log(2\pi)$.  One sees eleven examples of these asymptotically trivial zeros to the left of the critical line on the right side of Figure \ref{F:2}.  

There is a double pole of 
\[
-\frac{\pi^2}{4}\sec^2(\pi (1-s)/2)+\psi^\prime(1-s)
\]
at the negative even integers.  And \eqref{Eq:otoh} implies that as $s\to\infty$ with $\arg(s)$ a constant  $\pi/2-\delta$, $\arg(\log(\zeta(s))^{\prime\prime})$ is asymptotically constant (in fact, asymptotic to $\delta$).  For each double pole arising from a negative even integer, $\nu(s)$ will have by the Argument Principle  a pair of complex conjugate zeros inside the rays $\arg(s)=\pi\pm\delta$.  We refer to these zeros as \textsc{asymptotically trivial of the second kind}.  Examples in the upper half plane can be seen at the bottom of the left side of Figure \ref{F:2}; more examples can be seen in Figure \ref{F:wide}.  It would be interesting to understand the asymptotic behavior of the imaginary part of these zeros.

\begin{figure}[b]
\begin{center}
\includegraphics[scale=.8, viewport=0 0 500 80,clip]{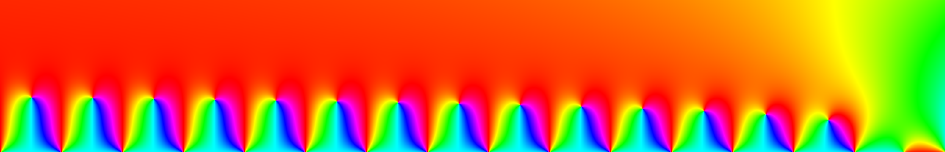}
\caption{Argument of $\log(\zeta(s))^{\prime\prime}$ in the  region $-30\le \sigma\le 1$, and $0\le t\le 5$. }\label{F:wide}
\end{center}
\end{figure}

\subsection*{Zero free region}
From the general theory of Dirichlet series, $\nu(s)$ has a right half plane free of zeros.
\begin{theorem}\label{Th:zerofree}
For $\tre(s)\ge 4.25$, we have that $\nu(s)\ne 0$.
\end{theorem}
\begin{remark}
\emph{Mathematica} shows there is a zero near $s=3.494+23.285i$.
\end{remark}
\begin{proof}
We have by the triangle inequality
\[
\left|\nu(s)\right|\ge \frac{a(2)}{2^\sigma}-\sum_{n=3}^\infty \frac{a(n)}{n^\sigma}
\]
Via summation by parts and the fact that
\[
\lim_{y\to\infty}A(y)y^{-\sigma}=0,
\]
we deduce that, with parameter $x$ to be determined,
\[
\left|\nu(s)\right|\ge
\frac{a(2)}{2^\sigma}-\sum_{n=3}^{x}\frac{a(n)}{n^\sigma}
+\frac{A(x)}{x^\sigma}-\sigma \int_{x}^\infty A(t)t^{-\sigma-1}\, dt.
\]
Via \eqref{Eq:1} it will suffice that we satisfy the two inequalities
\begin{gather*}
\frac{a(2)}{2^\sigma}-\sum_{n=3}^x \frac{a(n)}{n^\sigma}>\frac{1.5}{x^{\sigma/2}},\qquad\text{and}\\
\frac{A(x)}{x^\sigma}-\sigma \int_{x}^\infty p(\log(t))t^{-\sigma}\, dt-\left|10\cdot \sigma \int_{x}^\infty \log(t)^2 t^{-\sigma-1/2}\, dt\right|>-\frac{1}{x^{\sigma/2}}.
\end{gather*}
Once $x>4$ is fixed,
$
a(2)-\sum_{n=3}^x  a(n) \left(2/n\right)^\sigma
$ is an increasing function of $\sigma$, bounded above by $a(2)$, and $(2/\sqrt{x})^\sigma$ is decreasing to $0$.  Thus if the first inequality holds at $\sigma_0$, it will hold on the interval $[\sigma_0,\infty)$.

Next observe
\begin{multline*}
\sigma \int_{x}^\infty p(\log(t))t^{-\sigma}\, dt=\\
x^{-\sigma}\left(x\cdot p(\log(x))+\frac{q_1}{\sigma-1}+\frac{q_2}{(\sigma-1)^2}+\frac{q_3}{(\sigma-1)^3}+\frac{q_4}{(\sigma-1)^4}\right),
\end{multline*}
where the $q_j$ are certain polynomials in $x$ and $\log(x)$ in terms of the Stieltjes constants, positive for $x\ge4$.  
Meanwhile
\begin{multline*}
10\cdot \sigma \int_{x}^\infty \log(t)^2 t^{-\sigma-1/2}\, dt=\\
x^{1/2-\sigma}
\left(10
   \log (x)^2+\frac{r_1}{\sigma -1/2}+\frac{r_2}{\left(\sigma -1/2\right)^2}+\frac{r_3}{\left(\sigma -1/2\right)^3}\right),
\end{multline*}
for certain $r_i$, polynomials  in $\log(x)$ with positive coefficients.
Thus our second inequality is equivalent to
\begin{multline*}
x^{\sigma/2}>x\cdot p(\log(x))+10x^{1/2}\log(x)^2-A(x)\\
+x^{1/2}\left(\sum_{j=1}^4 \frac{ q_j}{(\sigma-1)^j}+\sum_{i=1}^3 \frac{ r_i}{(\sigma-1/2)^i}\right).
\end{multline*}
For fixed $x\ge 4$, the left side is increasing in $\sigma$, and the right side is decreasing in $\sigma$ so again this will hold on an interval $[\sigma_0,\infty)$.  With $x=40$, a calculation verifies that $\sigma_0=4.25$ suffices.  Furthermore, we deduce that for $\sigma>4.25$,
\begin{equation}\label{Eq:LB}
\frac{a(2)}{2^\sigma}-\sum_{n=3}^\infty \frac{a(n)}{n^\sigma}>\frac{.5}{40^{\sigma/2}}.
\end{equation}
\end{proof}

\subsection*{The number of zeros for $\nu(s)$}
Let 
\[
N_\nu(T)=\sharp\left\{\rho\,|\,\nu(\rho)=0, 0<\tim(\rho)<T, -4<\tre(\rho)\right\}.
\]
This count excludes the two flavors of asymptotically trivial zeros described above, except for a $O(1)$ error.
\begin{theorem}
\[
N_\nu(T)=2\left(\frac{T}{2\pi}\log\left(\frac{T}{2\pi}\right)-\frac{T}{2\pi}\right)-\frac{\log(2)}{\pi}T+O\left(\log(T)\right).
\]
\end{theorem}
\begin{proof}  Let $C$ be the boundary (described positively) of the rectangle with vertices $5+i10$, $5+iT$, $-4+iT$, $-4+i10$.  There are no asymptotically trivial zeros inside $C$.  By the functional equation and the zero free region, the nontrivial zeros are inside $C$.  By the Argument Principle, we need to estimate
\begin{multline*}
\frac{1}{2\pi i} \int_C \frac{d}{ds}\log(\nu(s))\, ds=\\
\frac{1}{2\pi i}\left\{\int_{-4+i10}^{5+i10} +\int_{5+i10}^{5+iT} +\int_{5+iT}^{-4+iT} +\int_{-4+iT}^{-4+i10} \right\}\frac{d}{ds}\log(\nu(s))\, ds\\
=\frac{1}{2\pi i}\left(I_1+I_2+I_3+I_4\right).
\end{multline*}
The integral $I_1$ is $O(1)$.  Next, $I_2$ is equal
\begin{equation}\label{Eq:temp}
\left.\log\left(\frac{a(2)}{2^s}\right)\right|_{5+i10}^{5+iT}+\left.\log\left(1+\sum_{n=3}^\infty\frac{a(n)}{a(2)}\left(\frac{2}{n}\right)^s\right)\right|_{5+i10}^{5+iT}.
\end{equation}
Via \eqref{Eq:LB}, we see that
\begin{equation}\label{Eq:13}
1-\sum_{n=3}^\infty\frac{a(n)}{a(2)}\left(\frac{2}{n}\right)^{-5}>.0025,
\end{equation}
thus
\begin{equation}\label{Eq:14}
\tre\left(1+\sum_{n=3}^\infty\frac{a(n)}{a(2)}\left(\frac{2}{n}\right)^{5+it}\right)>0,
\end{equation}
and the argument of the expression inside the second logarithm in \eqref{Eq:temp} is bounded by $\pm \pi/2$.  From the contribution of the first logarithm in \eqref{Eq:temp} we deduce that $I_2=-i\log(2)T+O(1)$.  Via a fairly routine argument based on Jensen's Theorem\footnote{as in, for example, \cite{Berndt}}, one sees that $I_3=O(\log(T))$.

Finally, for 
\[
I_4=\int_{-4+iT}^{-4+i10} \frac{d}{ds}\log(\nu(s))\, ds=\int_{-4+iT}^{-4+i150} \frac{d}{ds}\log(\nu(s))\, ds +O(1)
\]
 we will use the functional equation \eqref{Eq:FE} in the form
\begin{multline}
\nu(s)=\\
\chi^2(s)\nu(1-s) \left(1
+\left(\psi^\prime(1-s)-(\pi/2)^2\csc(\pi s/2)^2 \right)\frac{ \zeta (1-s)^2}{\nu(1-s)}\right).
\end{multline}
We observe that for $t\ge 150$
\begin{equation}\label{Eq:est1}
\left|\psi^\prime(5-it)-(\pi/2)^2\csc(\pi (4+it)/2)^2\right|< 1/140
\end{equation}
by the exponential decay of cosecant  and the Stirling's formula asymptotic for $\psi^\prime(5-it)$.  Also
\begin{gather}
\left|\log(\zeta(5-it))^{\prime\prime}\right|\ge\frac{\log(2)^2}{2^5}-\sum_{n=3}^\infty\frac{\Lambda(n)\log(n)}{n^5}\ge 0.0075,\notag\\
\left|\frac{ \zeta (5-it)^2}{\nu(5-it)}\right|\le \frac{1}{0.0075}<135.\label{Eq:est2}
\end{gather}
The product of \eqref{Eq:est1} and \eqref{Eq:est2} is $<1$ in absolute value, and thus
\[
\tre\left(
1
+\left(\psi^\prime(5-it)-(\pi/2)^2\csc(\pi (4+it)/2)^2 \right)\frac{ \zeta (5-it)^2}{\nu(5-it)}
\right)>0,
\]
and the argument of this expression is bounded between $-\pi/2$ and $\pi/2$.
This implies that on the vertical line $-4+it$, $T\ge t\ge 150$,
\[
\tim(\log(\nu(s)))=\tim(\log(\chi^2(s)\nu(1-s) ))+O(1).
\]
And similarly, via \eqref{Eq:13} and \eqref{Eq:14} we deduce that on this line,
\[
\tim(\log(\nu(s)))=\tim(\log(\chi^2(s)\log(2)^2 2^{s-1} ))+O(1).
\]
Via Stirling's formula
\begin{gather*}
\left.\arg(\chi^2(s))\right|_{-4+iT}^{-4+i150}=2T\log(T/2\pi)-2T+O(1),\quad\text{while}\\
\left.\arg(\log(2)^2 2^{s-1})\right|_{-4+iT}^{-4+i150}=-\log(2)T,\quad\text{so}\\
\tim(I_4)=2T\log(T/2\pi)-2T-\log(2)T+O(1).
\end{gather*}
\end{proof}
\subsection*{Zero density results}
\begin{proposition}  For $p(t)$ as before
\begin{equation}\label{Eq:7}
A(x)=
x\cdot p(\log(x))+O_\epsilon \left(x^{1/3+\epsilon}\right).
\end{equation}
\end{proposition}
\begin{proof}
Starting with \eqref{Eq:FE} and \eqref{Eq:6}, the proof very closely follows the $k=2$ case of the error estimates for the divisor function, as in Titchmarsh  \cite[Theorem 12.2]{Tit}.
\end{proof}
%The plot of $A(x)-x\cdot p(\log(x))$ for $1\le x\le 10^4$ given in Figure \ref{F:1} leads one to wonder  how much below this one can get the error bound.  The question is of interest because of the connection to zero density results; see Theorem \ref{Th:2}.

%\begin{figure}
%\begin{center}
%\includegraphics[scale=1, viewport=0 0 400 230,clip]{voronoi}
%\caption{$A(x)-x p(\log(x))$ v. $\pm x^{1/3}$.}\label{F:1}
%\end{center}
%\end{figure}

\begin{proposition}  Let
\[
\phi(s)=\left(1-2^{1-s}\right)^4\nu(s)
\]
The abscissa of convergence $\sigma_c$ for the series defining $\phi(s)$ is $\le 1/3$.
\end{proposition}
\begin{proof}  
The Dirichlet series expansion of $\phi(s)$ is $\sum_n b(n)n^{-s}$, where, if $2^j||n$,
 \[
b(n)=\sum_{m=0}^{\min(4,j)}\binom4m(-2)^ma(n/2^m).
 \]
With $B(x)=\sum_{n\le x}b(n)$, we have that
\begin{align*}
B(x)=&\sum_{m=0}^4\binom4m(-2)^m\sum_{\substack{k\le x\\2^m|k}}a(k/2^m)\\
=&\sum_{m=0}^4\binom4m(-2)^m\sum_{n\le x/2^m }a(n).
\end{align*}
Via \eqref{Eq:7} we see that
\begin{multline*}
B(x)=\sum_{m=0}^4\binom4m(-2)^m\left(\frac{x}{2^m} \cdot p(\log(x)-m\log(2))+O_\epsilon\left(x^{1/3+\epsilon}\right)\right)\\
=x\cdot \sum_{m=0}^4\binom4m(-1)^m p(\log(x)-m\log(2))+O_\epsilon\left(x^{1/3+\epsilon}\right).
\end{multline*}
With the shift operator $Ep(t)=p(t-\log(2))$ and difference operator $\Delta p=(I-E)p$, the main term is $x\cdot \Delta^4 p(\log(x))=0$, as $p$ has degree three and $\Delta$ reduces the degree.  Thus
\[
B(x)=O_\epsilon\left(x^{1/3+\epsilon}\right),
\]
and so for every $\epsilon>0$,
\begin{multline*}
\limsup_{x\to\infty} \frac{\log\left|B(x)\right|}{\log(x)}\le\\
\limsup_{x\to\infty} \frac{(1/3+\epsilon)\log(x)+\log(C(\epsilon))}{\log(x)}\le\frac{1}{3}+\epsilon,
\end{multline*}
and so by \cite[Theorem 1.3]{MV}, we obtain $\sigma_c\le 1/3$. 
\end{proof}

\begin{theorem}\label{Th:2}
If for positive $\delta$ we denote by $N_{5/6+\delta}(T)$ the number of zeros of $\nu(s)$ in the region $|\tim(s)|\le T$, $5/6+\delta\le\tre(s)$, then
\[
N_{5/6+\delta}(T)\ll_\delta T.
\]
\end{theorem}
\begin{proof}
The zeros of  $\nu(s)$ coincide with the zeros  of  $\phi(s)$.   We will imitate the proof of \cite[Theorem 6.18]{Nark}.  For $x_0>4.25$, and any integer $m$, set $K_{r,m}$ to be the circle with center $s_0=x_0+(1/2+m)i$ and radius $r=|x_0-5/6-\delta+i/2|$.  The circle passes through $5/6+\delta+mi$ and $5/6+\delta+(m+1)i$.  Increasing $x_0$ if necessary, the circle lies to the right of the line $\tre(s)=5/6+\delta/2$.  Set $K_{R,m}$ to be the circle with center $s_0=x_0+(1/2+m)i$ and radius $R=x_0-5/6-\delta/2$.
Finally let
\[
A=A(x_0)=2\inf_{\tre(s)=x_0}\left|\phi(s)\right|.
\]
The proof of Theorem \ref{Th:zerofree} implies $A>0$.  Now \cite[Corollary 2, p.260]{Nark}, a corollary to Jensen's Theorem,  implies there exists $C=C(r,R,A)$ such that the number of zeros of $\phi(s)$ in the rectangle 
\[
5/6+\delta\le \tre(s)\le x_0,\quad m<\tim(s)\le m+1
\]
does not exceed
\begin{multline*}
C\cdot \iint_{K_{R,m}}\left|\phi(x+iy)\right|^2\, dxdy\le \\
C\cdot \int_{5/6+\delta/2}^{x_0+R}\int_{m+1/2-R}^{m+1/2+R} \left|\phi(x+iy)\right|^2\, dy\,dx.
\end{multline*}
Summing over integers $m\in [-T,T]$ we deduce that
\[
N_{5/6+\delta}(T)=O\left(\int_{5/6+\delta/2}^{x_0+R}\int_{-T+1/2-R}^{T+1/2+R} \left|\phi(x+iy)\right|^2\, dy\,dx\right).
\]
From \cite[Corollary, p. 315]{Nark}, we deduce that 
\[
\int_{5/6+\delta/2}^{x_0+R}\int_{-T+1/2-R}^{T+1/2+R} \left|\phi(x+iy)\right|^2 \, dy\, dx\ll_\delta T.
\]
\end{proof}
\begin{remark}
The referee points out a mistake in the proof of \cite[Corollary, p. 315]{Nark}, and supplied a correction.  In the notation of that source, for $x\ge 1/2+\epsilon$, we have $2x-\epsilon>1+\epsilon$ so that $g(2x-\epsilon+it)$ converges absolutely.  This is all the proof requires, not the reference to Bohr and uniform convergence.
\end{remark}

\subsection*{Appendix: Numerical methods}
The graphics in Figures \ref{F:2} and \ref{F:wide} require the numerical computation of $\zeta(s)\zeta^{\prime\prime}(s)-\zeta^\prime(s)^2$ on a large grid of points in the complex plane.
Numerical computation of derivatives of a function $f(x)$ is often done by a method called Richardson extrapolation \cite[\S 5.7]{recipes}.  One has that
\begin{gather*}
\frac{f(x+h)-f(x-h)}{2h}=f^\prime(x)+\frac16f^{(3)}(x)h^2+O\left(h^4\right),\\
\frac{f(x+2h)-f(x-2h)}{4h}=f^\prime(x)+\frac23f^{(3)}(x)h^2+O\left(h^4\right),\\
\end{gather*}
so an appropriate linear combination of the left sides of the two equations computes $f^\prime(x)$ up to an error $O(h^4)$.  This can be readily generalized to computing each value on a rectangular grid of points of $\zeta(s)\zeta^{\prime\prime}-\zeta^\prime(s)^2$, up to an error $O(h^8)$,  with (asymptotically) a single evaluation of $\zeta(s)$.   One uses the saved function values at $\zeta(s\pm h)$, $\zeta(s\pm i h)$, $\zeta(s+(\pm h\pm i h))$, as well as $\zeta(s)$, and the solution to a linear system of 9 equations in 9 unknowns.

\subsubsection*{Acknowledgements}   Thanks to the anonymous referee for  careful reading of the manuscript and numerous helpful suggestions.

\end{document}